\documentclass[11pt]{article}
\usepackage{amssymb,amsfonts,amsmath,amsthm,cite, color}
\usepackage{epsfig}
\newtheorem{thm}{Theorem}[section]
\newtheorem{cor}{Corollary}
\newtheorem{lem}[thm]{Lemma}
\newtheorem{prop}[thm]{Proposition}

\makeatletter \@addtoreset{equation}{section}

\def\qed{\hfill \rule{4pt}{7pt}}
\def\pf{\noindent{\it Proof.\ }}
\newcommand{\@mathaccent}[2]{\newcommand{#1}[1]{\mathop{##1}\limits^{#2}}}%

\title{\bf\Large An Operator Approach to the Al-Salam-Carlitz  Polynomials}
\author{William Y. C. Chen$^{1}$,
Husam L. Saad$^{2}$, and Lisa H. Sun$^{3}$}
\date{$^{1,3}$Center for Combinatorics, LPMC-TJKLC\\
 Nankai University\\
Tianjin 300071, P.R. China\\
\vskip 0.2 cm $^2$Department of Mathematics, College of Science,\\
Basrah University\\  Basrah, Iraq\vskip 0.2 cm $^1$chen@nankai.edu.cn,
$^2$hus6274@hotmail.com, $^3$sunhui@nankai.edu.cn}

\begin{document}
\maketitle
\noindent \textbf{Abstract.} We present an operator approach to
Rogers-type formulas and Mehler's formulas for the Al-Salam-Carlitz
polynomials $U_n(x,y,a;q)$. By using the $q$-exponential operator,
we obtain a Rogers-type formula which leads to a linearization
formula. With the aid of a bivariate augmentation operator, we get a
simple derivation of  Mehler's formula due to by Al-Salam and
Carlitz, which requires  a terminating condition
 on a ${}_3\phi_2$ series. By means of the
Cauchy companion augmentation operator, we obtain Mehler's formula in a
similar form, but it does not need the terminating condition. We also give several identities
on the generating functions for products of the Al-Salam-Carlitz
polynomials which are extensions of formulas for Rogers-Szeg\"o
polynomials.



\section{Introduction}\label{sec1}

The Al-Salam-Carlitz polynomials are $q$-orthogonal polynomials
which arise in many applications such as the $q$-harmonic
oscillator, theta functions, quantum groups and coding theory; see
for example \cite{asksus93,akkwil85,gasrah,kim97}.  This paper presents
an operator approach to the Rogers-type formulas and
Mehler's formulas for the Al-Salam-Carlitz polynomials. These
 polynomials are a  generalization of the classical
Rogers-Szeg\"o polynomials which have been extensively studied, see
for example \cite{AtNa94,cao,car72,CG,chenss07}. There are two
classical formulas concerning the Rogers-Szeg\"o polynomials,
namely, Mehler's formula and the Rogers formula, in connection with
the Poisson kernel formula and the linearization formula. 

It is
natural to study the Rogers-type formulas and Mehler's formulas
beyond the Rogers-Szeg\"o polynomials. In fact,  Mehler's formula
for the Al-Salam-Carlitz polynomials has been derived by Al-Salam
and Carlitz \cite{alcar65}, which requires a terminating condition on
a $_3\phi_2$ series as mentioned by   Askey and Suslov
\cite{asksus93}. Using our operator approach, we deduce a new
formula  in a similar form, but it does not involve the terminating
condition. We also derive some Rogers-type formulas, one of which
leads to a linearization formula. In addition, we obtain several
identities on the generating functions of  products of the Al-Salam-Carlitz polynomials as
extensions of the formulas for the Rogers-Szeg\"{o} polynomials.

 We adopt the common notation on $q$-series
in Gasper and Rahman \cite{gasrah}. The set of integers is denoted
by $\mathbb{Z}$. Throughout this paper, $q$ is a fixed nonzero
complex number with $|q| < 1$. The $q$-shifted factorial is defined
for any complex parameter $a$ by
\begin{equation*}
(a;q)_\infty=\prod_{k=0}^\infty (1-aq^k)\quad \mbox{and} \quad
(a;q)_{n}=\frac{(a;q)_\infty}{(aq^n;q)_\infty},\quad n\in
\mathbb{Z}.
\end{equation*}
We shall  use the compact notation
\begin{align*}
 (a_{1}, a_{2}, \ldots, a_{m};q)_{n}&=(a_{1};q)_{n}(a_{2};q)_{n}
 \cdots(a_{m};q)_{n}, \quad \mbox{for}\  n\in
\mathbb{Z}\ \mbox{or} \  n=\infty.
\end{align*}
The $q$-binomial coefficient is given by
\[
{n \brack k}=\frac{(q;q)_{n}}{(q;q)_{n-k}(q;q)_{k}}. \] The basic
hypergeometric series $_{r}\phi_{s}$ are defined as follows,
\begin{align}\label{gHei}
{}_{r}\phi_{s} \left[\begin{array}{c}
  a_{1}, a_{2},
\ldots, a_{r}\\
 b_{1}, b_{2}, \ldots,
b_{s}\\
  \end{array};q, x \right]=\ \sum_{n=0}^{\infty}\frac{(a_{1},a_{2},\ldots,a_{r};q)_{n}}
  {(q,b_{1},b_2,\ldots,b_{s};q)_{n}}\ \left[(-1)^{n}q^{n\choose 2}\right]^{1+s-r}
  x^{n}.
\end{align}

This paper is primarily concerned with the Al-Salam-Carlitz
polynomials $U_n(x, y, a;q)$ which can be defined in terms of a
$_2\phi_1$ series
\begin{equation}\label{mbigqjac}
    U_n(x,y,a;q)=(-a)^nq^{n
\choose 2} {}_{2}\phi_1\left(\begin{array}{c}
q^{-n}, y/x\\
0
\end{array};q,\frac{qx}{a}\right).
\end{equation}
The following generating function for the the Al-Salam-Carlitz polynomials has been
given by Al-Salam and Carlitz \cite{alcar65},
\begin{align}\label{gf-mbjp}
\sum_{n=0}^\infty  U_n(x,y,a;q)
\frac{t^n}{(q;q)_n}=\frac{(at,yt;q)_\infty}{(xt;q)_\infty},
\end{align}
where $|xt|<1$. Since that the right-hand side of (\ref{gf-mbjp}) is symmetric in $a$ and $y$, the
polynomials $U_n(x,y,a;q)$ are symmetric in $a$ and $y$, that is,
\begin{equation}\label{uxy}
U_n(x,y,a;q)=U_n(x,a,y;q).
\end{equation}
This symmetry property will be used later.

In terms of  of the Cauchy polynomials
\[
P_n(x,y)=(x-y)(x-qy)\cdots(x-q^{n-1}y), \] with the generating
function
\begin{equation}\label{gf-pn}
\sum_{n=0}^\infty P_n(x,y)
\frac{t^n}{(q;q)_n}=\frac{(yt;q)_\infty}{(xt;q)_\infty}, \quad
|xt|<1,
\end{equation}
the Al-Salam-Carlitz polynomials can be expressed as
\begin{align}\label{mbqjp}
U_n(x,y,a;q)=\sum_{k=0}^n {n\brack k} (-1)^k q^{k\choose 2} a^k
P_{n-k}(x,y).
\end{align}
The above definition is
essentially the same as the original definition of the
Al-Salam-Carlitz polynomials $u_n^{(a)}(x;q)$,
\begin{align}\label{def-un}
u_n^{(a)}(x;q)&=(-a)^nq^{n \choose 2}
{}_{2}\phi_1\left(\begin{array}{c}
q^{-n}, x^{-1}\\
0
\end{array};q,\frac{qx}{a}\right).
\end{align}
Clearly, we have the following relation
\begin{align}
U_n(x,y,a;q)&=y^nu_n^{(a/y)}(x/y;q).\label{rebu}
\end{align}

The Al-Salam-Carlitz  polynomials are related to several $q$-orthogonal polynomials, such as the
$q$-Bessel polynomials $B_n(x,b;q)$ due to Abdi \cite{abdi65}, and
 the Stieltjes-Wigert polynomials $S_n(x;q)$
\cite[p.\,116]{koeswa94}. In particular, the Al-Salam-Carlitz
polynomials are connected to the bivariate Rogers-Szeg\"{o}
polynomials \cite{chenfuzhang}
$$
h_n(x,y|q)=\sum_{k=0}^n {n\brack k}P_k(x,y), $$ which have  the
generating function
\begin{equation}\label{gf-hxy}
\sum_{n=0}^{\infty}h_n(x,y|q)\frac{t^n}{(q;q)_n}=
\frac{(yt;q)_\infty}{(t,xt;q)_\infty}, \quad |t|<1, |xt|<1.
\end{equation}

On the other hand, although the Al-Salam-Carlitz polynomials can be expressed
in terms of the bivariate Rogers-Szeg\"o polynomials
\begin{equation}\label{reluh}
U_n(x,y,a;q)=(-1)^n q^{n\choose 2}a^n
h_n\Big(\frac{y}{a},\frac{x}{a}\Big|q^{-1}\Big),
\end{equation}
 as noted by Carlitz \cite{car72}, it is often happens that
an infinite $q$-series identity
no longer holds when $q$ is replaced by $q^{-1}$. In
fact, it turns out to be the case for the Rogers formula and
Mehler's formula for the polynomials $h_n(x,y|q)$. This suggests
that there is a need for a direct approach to deal with the
Al-Salam-Carlitz polynomials, and it is our hope   to
serve this purpose.

This paper is organized as follows.  In Section~2, we give an
overview of the $q$-exponential operator $T(bD_{q})$, and derive a
Rogers-type formula for $U_n(x,y,a;q)$ which leads to a
linearization formula. In Section~3, we construct a homogeneous
$q$-shift operator $\mathbb{F}(aD_{xy})$ and apply it to derive  Mehler's formula under the terminating condition. In Section~4, we
make use of
 the Cauchy companion operator to obtain two Rogers-type formulas and
  Mehler's formula without the terminating condition. In the last
 section,
 we provide four generating function identities
 for products of $U_n(x,y,a;q)$.

\section{A Rogers-Type Formula}

In this section, we give a Rogers-type formula for the
Al-Salam-Carlitz  polynomials $U_n(x,y,a;q)$ by using the
$q$-exponential operator $T(bD_{q})$. As a consequence, we obtain a linearization
formula for $U_n(x,y,a;q)$.

The $q$-differential operator, or the $q$-derivative, acting
on the variable $a$, is defined by
$$
D_q\{f(a)\}=\frac{f(a)-f(aq)}{a}.
$$
 The $q$-shift operator, denoted by $\eta$,  is defined by
\[
\eta\{f(a)\}=f(aq) \quad \mbox{and}\quad
\eta^{-1}\{f(a)\}=f(aq^{-1}), \] see, for example, \cite{andrews71,Roman85}. The operator  $\theta$ is defined by 
\begin{equation}\label{theta}
\theta=\eta^{-1} D_q, 
\end{equation} see Roman \cite{Roman85}.
Recall the  $q$-Leibniz rule for $D_q$, see \cite{Roman85}, 
\[
D_q^n \{f(a)g(a)\}=\sum_{k=0}^n {n\brack k}q^{k(k-n)}
D_q^k\{f(a)\}D_q^{n-k} \{g(q^ka)\}.
\]
 By convention, $D_q^0$ is understood as the identity, that is,
$D_q^0 \{f(a)\}=f(a)$.
Chen and Liu
\cite{chenliuII97} introduced the following two $q$-exponential operators
$$
T(bD_q)=\sum_{n=0}^\infty \frac{(bD_q)^n}{(q;q)_n} \quad \mbox{and}
\quad E(b\theta)=\sum_{n=0}^\infty \frac{q^{n\choose
2}(b\theta)^n}{(q;q)_n}
$$
for proving basic hypergeometric identities from their special cases. This method is
called parameter augmentation. The following lemma for the
$q$-exponential operator $T(bD_q)$  is easy to verify.

\begin{lem}
 We have
\begin{equation}\label{tbdq}
T(bD_q) \{a^n\}=\sum_{k=0}^n {n\brack k}b^ka^{n-k}.
\end{equation}
\end{lem}

From the $q$-Leibniz rule for $D_q$, Zhang and Wang
\cite{zhangwang05}  obtain  the following identity.

\begin{lem}
Let $n$ be a nonnegative integer. Then
\[
D_q^n\left\{\frac{(at;q)_\infty}{(av;q)_\infty}\right\}=v^n
(t/v;q)_n \frac{(atq^n;q)_\infty}{(av;q)_\infty}.
\]
\end{lem}

Based on the above relation, we  obtain the following formula.

\begin{lem}\label{lemtbdq2/1} We have
\begin{equation}\label{tbdq2/1}
T(bD_q)
\left\{\frac{(as,at;q)_\infty}{(av;q)_\infty}\right\}=
\frac{(as,at;q)_\infty}{(av;q)_\infty}
\sum_{k=0}^\infty \frac{(-1)^kq^{k\choose
2}(av;q)_k(bs)^k}{(q;q)_k(as,at;q)_k}\, {}_{2}\phi_1\left(\begin{array}{c}
t/v,0\\
atq^k
\end{array};q,bv\right).
\end{equation}
\end{lem}
\begin{proof}
By the definition of $T(bD_q)$ and the $q$-Leibniz rule for $D_q$,
we have \allowdisplaybreaks
\begin{align*}
T(b & D_q) \left\{\frac{(as,at;q)_\infty}{(av;q)_\infty}\right\}=
\sum_{n=0}^\infty\frac{b^n}{(q;q)_n}
D_q^n \left\{\frac{(as,at;q)_\infty}{(av;q)_\infty}\right\}\\[5pt]
&=\sum_{n=0}^\infty\frac{b^n}{(q;q)_n}\sum_{k=0}^n {n\brack
k}q^{k(k-n)}D_q^k\{(as;q)_\infty\}D_q^{n-k}
\left\{\frac{(atq^k;q)_\infty}{(avq^k;q)_\infty}\right\}\\[5pt]
&=\sum_{n=0}^\infty\frac{b^n}{(q;q)_n}\sum_{k=0}^n
{n\brack k}q^{k(k-n)} (-1)^k q^{k\choose 2} s^k (asq^k;q)_\infty
(vq^k)^{n-k} (t/v;q)_{n-k}
\frac{(atq^n;q)_\infty}{(avq^k;q)_\infty}\\[5pt]
&=\frac{(as,at;q)_\infty}{(av;q)_\infty}\sum_{n=0}^\infty\frac{b^n}{(q;q)_n}
\sum_{k=0}^n
{n\brack k} (-1)^k q^{k\choose 2} s^k
\frac{(av;q)_k(t/v;q)_{n-k}v^{n-k}}{(as;q)_k(at;q)_n}\\[5pt]
&=\frac{(as,at;q)_\infty}{(av;q)_\infty} \sum_{k=0}^\infty
\frac{(-1)^kq^{k\choose 2}(av;q)_k(bs)^k}{(q,as,at;q)_k}
\sum_{n=0}^\infty \frac{ (t/v;q)_n(bv)^n}{(q;q)_n(atq^k;q)_n}\\[5pt]
&=\frac{(as,at;q)_\infty}{(av;q)_\infty} \sum_{k=0}^\infty
\frac{(-1)^kq^{k\choose
2}(av;q)_k(bs)^k}{(q;q)_k(as,at;q)_k}{}_{2}\phi_1\left(\begin{array}{c}
t/v,0\\
atq^k
\end{array};q,bv\right).
\end{align*}
This completes the proof.
\end{proof}

Now we are ready to present a Rogers-type formula for the
polynomials $U_n(x,y,a;q)$.

\begin{thm}\label{rog1} We have
\begin{align}\label{rogers1}
\sum_{n=0}^\infty &\sum_{m=0}^\infty U_{n+m}(x,y,a;q)
\frac{t^n}{(q;q)_n}\frac{s^m}{(q;q)_m}\nonumber\\
&=\frac{(as,ys;q)_\infty}{(xs;q)_\infty}\sum_{k=0}^\infty
\frac{(-1)^kq^{k\choose
2}(xs;q)_k(at)^k}{(q;q)_k(as,ys;q)_k}{}_{2}\phi_1\left(\begin{array}{c}
y/x, 0\\
ysq^k
\end{array};q,xt\right),
\end{align}
provided that $|xs|<1$.
\end{thm}

\begin{proof}
Setting $m\rightarrow m-n$, exchanging the order of the sum on the
left hand side of \eqref{rogers1}, and applying the operator
identity (\ref{tbdq}), we obtain \allowdisplaybreaks
\begin{align*}
\sum_{n=0}^\infty & \sum_{m=0}^\infty U_{n+m}(x,y,a;q)
\frac{t^n}{(q;q)_n}\frac{s^m}{(q;q)_m}\\
&=\sum_{n=0}^\infty \sum_{m=n}^\infty U_m(x,y,a;q)
\frac{t^n}{(q;q)_n}\frac{s^{m-n}}{(q;q)_{m-n}}=\sum_{m=0}^\infty
\frac{U_m(x,y,a;q)}{(q;q)_m} \sum_{n=0}^m {m\brack
n}t^ns^{m-n}\\
&=\sum_{m=0}^\infty \frac{U_m(x,y,a;q)}{(q;q)_m} T(tD_q)
 \{s^m\}=T(tD_q) \left\{\sum_{m=0}^\infty U_m(x,y,a;q)\frac{s^m}{(q;q)_m}\right\}\quad (|xs|<1) \\
&=T(tD_q) \left\{\frac{(as,ys;q)_\infty}{(xs;q)_\infty} \right\},
\end{align*}
where  $D_q$ acts on the parameter $s$.
Applying  Lemma \ref{lemtbdq2/1}, we complete the proof.
\end{proof}

From the above Rogers-type formula, we obtain the
following linearization formula for $U_n(x,y,a;q)$.

\begin{thm} For $n, m\geq 0$, we have
\[
U_{n+m}(x,y,a;q)=\sum_{k=0}^n {n\brack k} (-1)^kq^{k\choose
2}(aq^m)^k P_{n-k}(x,y) U_m(x,yq^{n-k},a;q).
\]
\end{thm}

\begin{proof}
Rewrite the Rogers-type formula \eqref{rogers1} as follows
\begin{align*}
\sum_{n=0}^\infty &\sum_{m=0}^\infty U_{n+m}(x,y,a;q)
\frac{t^n}{(q;q)_n}\frac{s^m}{(q;q)_m}\nonumber\\
&=\sum_{k=0}^\infty \frac{(-1)^kq^{k\choose
2}(at)^k}{(q;q)_k}\sum_{n=0}^\infty \frac{(y/x;q)_n}{(q;q)_n} (xt)^n
\frac{(asq^k,ysq^{n+k};q)_\infty}{(xsq^k;q)_\infty}\\
&=\sum_{k=0}^\infty \frac{(-1)^kq^{k\choose
2}(at)^k}{(q;q)_k}\sum_{n=0}^\infty \frac{(y/x;q)_n}{(q;q)_n} (xt)^n
\sum_{l=0}^\infty U_l(x,yq^n,a;q) \frac{(sq^k)^l}{(q;q)_l}.
\end{align*}
Equating the coefficients of $t^{n}s^{m}$ in the above equation, the
desired identity follows.
\end{proof}

\section{Mehler's Formula}

In this section, we aim to introduce the homogeneous $q$-shift
operator which can be used to give a simple derivation of  Mehler's
formula for $U_n(x,y,a;q)$ due to Al-Salam and Carlitz.

Recall that the homogeneous $q$-difference operator $D_{xy}$ introduced by Chen, Fu and
Zhang \cite{chenfuzhang} is
given  by
\begin{equation}\label{opexy}
D_{xy} \{f(x,y)\}=\frac{f(x,q^{-1}y)-f(qx,y)}{x-q^{-1}y}.
\end{equation}
Based on this operator $D_{xy}$, we
construct the following homogeneous $q$-shift operator
\begin{equation}\label{operator}
\mathbb{F}(aD_{xy})=\sum_{n=0}^\infty \frac{(-1)^nq^{n\choose
2}(aD_{xy})^n}{(q;q)_n}.
\end{equation}

The $q$-difference operator $D_{xy}$ has the following basic
properties.

\begin{prop}\label{propdxy}
We have
\begin{align}
  D_{xy}^{k}\{P_n(x,y)\} &= \frac{(q;q)_{n}}{(q;q)_{n-k}}P_{n-k}(x,y),\label{dxypn} \\
  D_{xy}^{k} \left\{\frac{(yt;q)_\infty}{(xt;q)_\infty}\right\} &=
  t^{k} \frac{(yt;q)_\infty}{(xt;q)_\infty}.\label{dxyytxt}
\end{align}
\end{prop}

Invoking (\ref{dxypn}), the Al-Salam-Carlitz  polynomials
$U_n(x,y,a;q)$ can be expressed in terms of the homogeneous
$q$-shift operator $\mathbb{F}(aD_{xy})$.

\begin{thm} We have
\begin{equation}\label{fbdxypn}
  U_n(x,y,a;q) =  \mathbb{F}(aD_{xy}) \left\{P_n(x,y)\right\}.
\end{equation}
\end{thm}

By using (\ref{dxyytxt}), it is easy to derive the following
relation.

\begin{prop}\label{fbdxy} We have
\begin{equation}\label{fdxy}
\mathbb{F}(aD_{xy})\left\{\frac{(yt;q)_\infty}{(xt;q)_\infty}\right\}
=\frac{(at,yt;q)_\infty}{(xt;q)_\infty}.
\end{equation}
\end{prop}

Combining \eqref{fbdxypn} and
 \eqref{fdxy},  the
  generating function  of
$U_n(x,y,a;q)$ can be derived as follows,
\begin{align*}
\sum_{n=0}^\infty  U_n(x,y,a;q)
\frac{t^n}{(q;q)_n}&=\mathbb{F}(aD_{xy}) \left\{\sum_{n=0}^\infty
P_n(x,y)\frac{t^n}{(q;q)_n}\right\} \qquad (|xt|<1)\\
& =\mathbb{F}(aD_{xy})
\left\{\frac{(yt;q)_\infty}{(xt;q)_\infty}\right\}
=\frac{(at,yt;q)_\infty}{(xt;q)_\infty}.
\end{align*}

The following identity will be used to derive Mehler's formula.

\begin{thm}\label{cordxy} Assume $|xs|<1$. We have
\begin{align}\label{opefbd}
\mathbb{F}(aD_{xy})& \left\{ \frac{P_n(x,y)(ys;q)_\infty
}{(ys;q)_n(xs;q)_\infty
}\right\}=\frac{(ysq^n,as;q)_\infty}{(xs;q)_\infty}\sum_{k=0}^n
{n\brack k} (-1)^{k}q^{k\choose 2}
\frac{(xs;q)_k(y/x;q)_{n-k}}{(as;q)_k} x^{n-k}a^{k}.
\end{align}
\end{thm}

\begin{proof}
Applying (\ref{fbdxypn}), the left hand side of the Rogers-type
formula \eqref{rog1} equals
 \allowdisplaybreaks
\begin{align}
\mathbb{F}(a& D_{xy}) \left\{\sum_{n=0}^\infty \sum_{m=0}^\infty
P_{n+m}(x,y) \frac{t^n}{(q;q)_n}\frac{s^m}{(q;q)_m}\right\}\nonumber\\
&=\mathbb{F}(aD_{xy})\left\{\sum_{n=0}^\infty P_{n}(x,y)
\frac{t^n}{(q;q)_n}\sum_{m=0}^\infty
P_m(x,q^ny)\frac{s^m}{(q;q)_m}\right\}\nonumber \quad (|xs|<1)\\
&=\mathbb{F}(aD_{xy})\left\{\sum_{n=0}^\infty P_{n}(x,y)
\frac{t^n}{(q;q)_n}\frac{(yq^ns;q)_\infty}{(xs;q)_\infty}\right\}\nonumber\\
&=\sum_{n=0}^\infty \mathbb{F}(aD_{xy})\left\{\frac{P_{n}(x,y)
(ys;q)_\infty}{(ys;q)_n(xs;q)_\infty}\right\}
\frac{t^n}{(q;q)_n}.\label{fdxy1}
\end{align}
On the other hand, the right hand side of
\eqref{rog1} can be restated as
\begin{align}
&\frac{(as,ys;q)_\infty}{(xs;q)_\infty}\sum_{k=0}^\infty
\frac{(-1)^kq^{k\choose
2}(xs;q)_k(at)^k}{(q;q)_k(as,ys;q)_k}\sum_{l=0}^\infty
\frac{(y/x;q)_l}{(q,ysq^k;q)_l}(xt)^l. \label{fdxy2}
\end{align}
Equating the coefficients of $t^n$ in \eqref{fdxy1} and
\eqref{fdxy2}, we complete the proof.
\end{proof}

Applying the above operator identity, we  obtain Mehler's formula
involving a terminating ${}_3\phi_2$ series.

\begin{thm} We have
\begin{align}
\sum_{n=0}^\infty & (-1)^nq^{-{n\choose 2}}
U_{n}(x,y,a;q)U_{n}(u,v,b;q)
\frac{t^n}{(q;q)_n}\nonumber\\
&=\frac{(abt,ybt,avt;q)_\infty}{(xbt,aut;q)_\infty} {}_{3}\phi_2
\left(\begin{array}{c}
y/x,v/u, q/abt\\
q/xbt, q/aut
\end{array};q,q\right),\label{mehler1}
\end{align}
where $y/x=q^{-r}$ or $v/u=q^{-r}$ for a nonnegative
integer $r$, and
$\max\{|xbtq^{-r}|,|autq^{-r}|\}<1$.

\end{thm}

\begin{proof} Using (\ref{fbdxypn}), we find
\allowdisplaybreaks
\begin{align}
\sum_{n=0}^\infty &(-1)^n q^{-{n\choose 2}}
U_{n}(x,y,a;q)U_{n}(u,v,b;q)
\frac{t^n}{(q;q)_n}\nonumber\\
&=\mathbb{F}(aD_{xy}) \left\{\sum_{n=0}^\infty (-1)^n q^{-{n\choose
2}} P_{n}(x,y)U_{n}(u,v,b;q) \frac{t^n}{(q;q)_n}\right\}\nonumber\\
&=\mathbb{F}(aD_{xy}) \left\{\sum_{n=0}^\infty (-1)^n q^{-{n\choose
2}} P_{n}(x,y)\Bigg(\sum_{k=0}^n {n \brack k}(-1)^k q^{k\choose
2}b^kP_{n-k}(u,v)\Bigg) \frac{t^n}{(q;q)_n}\right\}\nonumber\\
&= \mathbb{F}(aD_{xy}) \left\{\sum_{n=0}^\infty
\frac{(-1)^nq^{-{n\choose
2}}P_n(u,v)P_n(x,y)t^n}{(q;q)_n}\sum_{k=0}^\infty P_k(x,q^ny)\frac{
(btq^{-n})^k}{(q;q)_k}\right\}\label{term1}.
\end{align}
The terminating  condition  $v/u=q^{-r}$ or $y/x=q^{-r}$ implies
that the first sum in (\ref{term1}) is finite.
 Utilizing  \eqref{gf-pn},  we see that (\ref{term1}) equals
\[
\sum_{n=0}^\infty \frac{(-1)^nq^{-{n\choose
2}}P_n(u,v)t^n}{(q;q)_n}\ \mathbb{F}(aD_{xy})
\left\{\frac{P_n(x,y)}{(ybtq^{-n};q)_n}\frac{(ybtq^{-n};q)_\infty}{(xbtq^{-n};q)_\infty}\right\}.
\]
Applying \eqref{opefbd} with $s\rightarrow
btq^{-n}$, the above sum equals
\begin{align}
\sum_{n=0}^\infty & \frac{(-1)^nq^{-{n\choose
2}}P_n(u,v)t^n}{(q;q)_n}
\frac{(ybt,abtq^{-n};q)_\infty}{(xbtq^{-n};q)_\infty}\sum_{k=0}^n
{n\brack k} (-1)^{k}q^{k\choose 2}
\frac{(xbtq^{-n};q)_k(y/x;q)_{n-k}}{(abtq^{-n};q)_k} x^{n-k}a^{k}
\nonumber\\[5pt]
&=\frac{(abt,ybt;q)_\infty}{(xbt;q)_\infty} \sum_{n=0}^\infty
\frac{(-1)^n q^{-{n\choose
2}}(abtq^{-n},y/x;q)_nP_{n}(u,v)(xt)^n}{(q;q)_n(xbtq^{-n};q)_n}\sum_{k=0}^\infty
P_k(u,vq^n)\frac{(atq^{-n})^k }{(q;q)_k},\label{term2}
\end{align}
Under  the terminating condition,
the above sum further simplifies to the right hand side of (\ref{mehler1}).
This completes the proof.
\end{proof}

We  remark that  the second sums in \eqref{term1} and \eqref{term2}
do not converge when $n$ tends to infinity. To avoid this problem,
we may restrict our attention to the case that $v/u=q^{-r}$ or $y/x=q^{-r}$, where $r$ is a
nonnegative integer, as noticed by   Askey and Suslov
\cite{asksus93} and Fang \cite{fang07}.

Owing to the symmetry property and the relation \eqref{rebu},
Mehler's formula for the Al-Salam-Carlitz polynomials
$u^{(a)}_n(x;q)$ given by Al-Salam and Carlitz \cite{alcar65} can be
recovered from the above theorem by setting $y/a\rightarrow a,
x/a\rightarrow x, v/b \rightarrow b, u/b\rightarrow y,
-abt\rightarrow t$.

\begin{cor}[Mehler's formula for $u^{(a)}_n(x|q)$]
\begin{align}\label{orimehler}
\sum_{n=0}^\infty &q^{-{n\choose 2}}
u^{(a)}_{n}(x;q)u^{(b)}_{n}(y;q)
\frac{t^n}{(q;q)_n}=\frac{(-t,-at,-bt;q)_\infty}{(-xt,-yt;q)_\infty}
{}_{3}\phi_2 \left(\begin{array}{c}
a/x, b/y,-q/t\\
-q/xt, -q/yt
\end{array};q,q\right),
\end{align}
where $a/x=q^{-r}$ or $b/y=q^{-r}$ for a nonnegative integer
$r$, and $\max\{|xtq^{-r}|,|ytq^{-r}|\}<1$.
\end{cor}

\section{The Cauchy Companion Operator}

In this section, we apply the Cauchy companion operator
defined by Chen \cite{chenyb} to derive the Rogers-type formulas and
Mehler's formula without the terminating condition.
Recall that the Cauchy  augmentation operator is defined by Chen
and Gu \cite{CG},
\[
T(a,b;D_q)=\sum_{n=0}^\infty \frac{(a;q)_n}{(q;q)_n} (bD_q)^n.
\]
Moreover, Chen \cite{chenyb} introduced the Cauchy companion operator
\begin{equation}\label{cauchyc}
E(a,b;\theta)=\sum_{n=0}^\infty \frac{(a;q)_n(-b\theta)^n}{(q;q)_n}.
\end{equation}

As observed by Chen \cite{chenyb}, when one applies $E(a,b;\theta)$
to the product $(cs,ct;q)_\infty/(cv;q)_\infty$, one does not get a
valid identity
 by directly using  $q$-Leibniz rule because of the
 convergence consideration. Instead,
  we may use the following expansion for $D_q^n$
\[
D_q^{n} \{f(c)\}=c^{-n}q^{-{n\choose 2}} \sum_{k=0}^n (-1)^k
{n\brack k}q^{n-k \choose 2} f(cq^k).
\]
In this way, we can deduce an alternative expansion
 for $E(a,b;\theta)$ which is convergent \cite{chenyb}
\begin{equation}\label{expansionE}
E(a,b;\theta) \{f(c)\}=\frac{(abq/c;q)_\infty}{(bq/c;q)_\infty}
\sum_{k=0}^\infty \frac{(a;q)_kf(cq^{-k})q^{k\choose
2}}{(q,abq/c;q)_k}\bigg(-\frac{bq}{c}\bigg)^k,
\end{equation}
where $|bq/c|<1$. Furthermore, we will be able to derive the
Rogers-type formulas and a Mehler's formula without the terminating
condition based on the following  operator identities
established by Chen \cite{chenyb}.

\begin{prop}\label{propE} Assume that
the operator $E(a,b;\theta)$ acts on the parameter $c$, then
\allowdisplaybreaks
\begin{align}
&E(a,b;\theta)\{c^n\}=\sum_{k=0}^n {n\brack k} (a;q)_k (-bq)^k
c^{n-k}q^{k\choose 2} q^{-nk},\quad (n\geq 0),\label{cauchy1}\\[8pt]
&E(a,b;\theta)\left\{\frac{(ct;q)_\infty}{(cv;q)_\infty}\right\}
=\frac{(ct;q)_\infty}{(cv;q)_\infty} {}_{2}\phi_1
\left(\begin{array}{c}
a, t/v\\
q/cv
\end{array};q,\frac{bq}{c}\right),\qquad (|bq/c|<1),\label{cauchy3}\\[8pt]
&E(a,b;\theta)\left\{\frac{(cs,ct;q)_\infty}{(cv;q)_\infty}\right\}
=\frac{(abq/c,cs,ct;q)_\infty}{(bq/c,cv;q)_\infty} {}_{3}\phi_2
\left(\begin{array}{c}
a, q/cs,q/ct\\
abq/c,q/cv
\end{array};q,\frac{bst}{v}\right),\nonumber \\[6pt]
&\qquad\qquad\qquad\qquad\qquad\qquad\qquad\qquad\qquad\qquad\qquad\qquad (\max\{|bq/c|,|bst/v|\}<1)
.\label{cauchy4}
\end{align}
\end{prop}

 In the light of the property \eqref{cauchy1}, we obtain the following
 operator representation of $U_n(x,y,a;q)$.

\begin{thm} Assume that the operator $E(y/x,x;\theta)$ acts on the parameter $a$, then
\begin{equation}\label{cauchyB}
E(y/x,x;\theta)\{(-1)^nq^{n\choose 2} a^n\}=U_n(x,y,a;q).
\end{equation}
\end{thm}

The above operator identity leads to another Rogers-type
formula for the Al-Salam-Carlitz  polynomials.

\begin{thm} We have
\allowdisplaybreaks
\begin{align}
&\sum_{n=0}^\infty \sum_{m=0}^\infty (-1)^{n}q^{-{n\choose
2}-nm}U_{n+m}(x,y,a;q) \frac{t^n}{(q;q)_n}\frac{s^m}{(q;q)_m}
=\frac{(as;q)_\infty}{(at;q)_\infty}{}_{2}\phi_1
\left(\begin{array}{c}
y/x,s/t\\
q/at
\end{array};q,\frac{xq}{a}\right), \label{Rogers15}
\end{align}
where $\max\{|at|,|xq/a|\}<1$.
\end{thm}

\pf By  \eqref{cauchyB}, the left hand side of
\eqref{Rogers15} can be written as
\begin{align*}
&\sum_{n=0}^\infty \sum_{m=0}^\infty (-1)^{n}q^{-{n\choose 2}-nm}
U_{n+m}(x,y,a;q)
\frac{t^n}{(q;q)_n}\frac{s^m}{(q;q)_m}\\
&\quad=E(y/x,x;\theta)\left\{\sum_{n=0}^\infty \sum_{m=0}^\infty
(-1)^{m}q^{m\choose 2}
\frac{a^nt^n}{(q;q)_n}\frac{a^ms^m}{(q;q)_m}\right\}\\
&\quad=E(y/x,x;\theta)\left\{\sum_{n=0}^\infty
\frac{a^nt^n}{(q;q)_n}\sum_{m=0}^\infty (-1)^m q^{m\choose
2}\frac{(as)^m}{(q;q)_m}\right\}\quad (|at|<1)\\
&\quad=E(y/x,x;\theta)\left\{\frac{(as;q)_\infty}{(at;q)_\infty}\right\}.
\end{align*}
Using \eqref{cauchy3}, we complete the proof.  \qed

Applying  the operator $E(a,b;\theta)$ one more time,  we obtain
following triple sum identity.

\begin{thm} We have
\begin{align}\label{triple}
&\sum_{n=0}^\infty\sum_{m=0}^\infty\sum_{k=0}^\infty
(-1)^{k}q^{-{k\choose 2}-(m+n)k-mn}U_{n+m+k}(x,y,a;q)
\frac{t^n}{(q;q)_n}\frac{s^m}{(q;q)_m}\frac{v^k}{(q;q)_k}\nonumber\\
&\qquad=\frac{(yq/a,as,at;q)_\infty}{(xq/a,av;q)_\infty}
{}_{3}\phi_2 \left(\begin{array}{c}
y/x,q/as,q/at\\
yq/a,q/av
\end{array};q,\frac{xst}{v}\right),
\end{align}
provided that $\max\{|av|,|xq/a|,|xst/v|\}<1$.
\end{thm}

\pf By  the operator identity \eqref{cauchyB},  the left
hand side of \eqref{triple} equals
\begin{align*}
&E(y/x,x;\theta)\left\{\sum_{n=0}^\infty\sum_{m=0}^\infty\sum_{k=0}^\infty
(-1)^{m+n}q^{{n\choose 2}+{m\choose 2}}
\frac{(at)^n}{(q;q)_n}\frac{(as)^m}{(q;q)_m}\frac{(av)^k}{(q;q)_k}\right\}\quad (|av|<1)\\
&\quad=E(y/x,x;\theta)\left\{\frac{(as,at;q)_\infty}{(av;q)_\infty}\right\}.
\end{align*}
Applying the operator identity
\eqref{cauchy4}, we complete the proof.  \qed

Setting $s\rightarrow 0$ and $t\rightarrow s, v\rightarrow t$ and
applying Jackson's transformation formula \cite[III.4]{gasrah}, the
triple sum \eqref{triple} reduces to the Rogers-type formula
\eqref{Rogers15}.
The Cauchy companion operator also applies to other Rogers-type
formulas for the Al-Salam-Carlitz polynomials, including the one
given in the previous section. Moreover, we can also derive the
following Rogers-type formula \begin{align*} \sum_{n=0}^\infty
\sum_{m=0}^\infty &
q^{-mn}U_{n+m}(x,y,a;q)\frac{t^n}{(q;q)_n}\frac{s^m}{(q;q)_m}
\\[3pt]
& =\frac{(yq/a,as,at;q)_\infty}{(xq/a;q)_\infty}{}_{3}\phi_1
\left(\begin{array}{c}
y/x,q/as,q/at\\
yq/a
\end{array};q,\frac{axst}{q}\right),
\end{align*}
where $\max\{|xq/a|, |axst/q|\}<1$. It should be noticed that the
above formula is not a consequence of the Rogers formula for the
bivariate Rogers-Szeg\"{o} polynomials $h_n(x,y|q)$
\cite[Theorem\,3.1]{chenss07} by replacing $q$ with $q^{-1}$.

We now present  Mehler's formula without the
 terminating condition.

\begin{thm} We have
\begin{align}\label{mehlerE}
&\sum_{n=0}^\infty (-1)^n q^{-{n\choose
2}}U_n(x,y,a;q)U_n(u,v,b;q)\frac{t^n}{(q;q)_n}
\nonumber\\
&\qquad=\frac{(yq/a,abt,avt;q)_\infty}{(xq/a,aut;q)_\infty}
{}_{3}\phi_2 \left(\begin{array}{c}
y/x,q/abt,q/avt\\
yq/a,q/aut
\end{array};q,\frac{xbvt}{u}\right),
\end{align}
provided that $\max\{|aut|,|xq/a|,|xbvt/u|\}<1$.
\end{thm}

\pf Using \eqref{cauchyB} and  \eqref{gf-mbjp}, we find
\begin{align*}
&\sum_{n=0}^\infty (-1)^n q^{-{n\choose
2}}U_n(x,y,a;q)U_n(u,v,b;q)\frac{t^n}{(q;q)_n}
\nonumber\\
&\quad=E(y/x,x;\theta)\left\{\sum_{n=0}^\infty
U_n(u,v,b;q)\frac{(at)^n}{(q;q)_n}\right\}\qquad (|aut|<1)\\
&\quad=E(y/x,x;\theta)\left\{\frac{(avt;abt;q)_\infty}{(aut;q)_\infty}\right\}.
\end{align*}
So the proof is completed by using the operator identity
\eqref{cauchy4}.  \qed

Comparing the above Mehler's formula with the terminating form
\eqref{mehler1}, it leads to the following transformation formula
for ${}_3\phi_2$ series.

\begin{cor} We have
\begin{align}\label{trans}
{}_{3}\phi_2 \left(\begin{array}{c}
y/x,q/abt,q/avt\\
yq/a,q/aut
\end{array};q,\frac{xbvt}{u}\right)=\frac{(ybt,xq/a;q)_\infty}{(xbt,yq/a;q)_\infty} {}_{3}\phi_2
\left(\begin{array}{c}
y/x, v/u, q/abt\\
q/xbt, q/aut
\end{array};q,q\right),
\end{align}
provided that $v/u=q^{-r}$ for a nonnegative integer $r$.
\end{cor}

\section{Generating Functions for Products of $U_n(x,y,a;q)$}

The objective of this section is to give several generating function
formulas for
products of the Al-Salam-Carlitz polynomials by using the
Cauchy companion operator.
Keep in mind that the
Al-Salam-Carlitz polynomials are extensions of the Rogers-Szeg\"{o}
polynomials defined by
\[
g_n(a|q) =\sum_{k=0}^n {n \brack k} q^{k(k-n)} a^k.
\]
It is easily seen that
\[
 U_n(0,1,a;q)=(-1)^nq^{n\choose 2} g_n(a|q).
 \]

\begin{thm} \label{novelgf} We have
\begin{align*}
&\sum_{n=0}^\infty\sum_{m=0}^\infty (-1)^{n+m}q^{-{n+m\choose 2}}
U_{n+m}(x,y,a;q)U_n(u,v,b;q)U_m(z,w,c;q) \frac{t^n}{(q;q)_n}
\frac{s^m}{(q;q)_m}\\[5pt]
&\quad=\frac{(yq/a,abt,avt,acs,aws;q)_\infty}{(xq/a,aut,azs;q)_\infty}
{}_{5}\phi_3 \left(\begin{array}{c}
y/x,q/abt,q/avt,q/acs,q/aws\\
yq/a,q/aut,q/azs
\end{array};q,\frac{xabcvwts}{uzq}\right),
\end{align*}
provided that $\max\{|xq/a|,|aut|,|azs|\}<1$.
\end{thm}

\pf By  the operator identity \eqref{cauchyB} acting on the
parameter $a$, we obtain
\begin{align}
&\sum_{n=0}^\infty\sum_{m=0}^\infty (-1)^{n+m}q^{-{n+m\choose 2}}
U_{n+m}(x,y,a;q)U_n(u,v,b;q)U_m(z,w,c;q) \frac{t^n}{(q;q)_n}
\frac{s^m}{(q;q)_m}\nonumber \\[5pt]
&\quad=E(y/x,x;\theta)\left\{\sum_{n=0}^\infty\sum_{m=0}^\infty
U_n(u,v,b;q)U_m(z,w,c;q) \frac{(at)^n}{(q;q)_n}
\frac{(as)^m}{(q;q)_m} \right\}. \label{ey}
\end{align}
Employing the generating function \eqref{gf-mbjp} with
$\max\{|aut|,|azs|\}<1$ and the operator identity \eqref{expansionE}
with $|xq/a|<1$, we see   that (\ref{ey})  equals
\begin{align*}
&E(y/x,x;\theta)\left\{\frac{(abt,avt;q)_\infty
(acs,aws;q)_\infty}{(aut;q)_\infty(azs;q)_\infty} \right\}\\[5pt]
&\quad=\frac{(yq/a;q)_\infty}{(xq/a;q)_\infty} \sum_{k=0}^\infty
\frac{(y/x;q)_kq^{k\choose 2}}{(q, yq/a;q)_k}
\Big(-\frac{xq}{a}\Big)^k
\frac{(abtq^{-k},avtq^{-k},
acsq^{-k},awsq^{-k};q)_\infty}{(autq^{-k},azsq^{-k};q)_\infty},
\end{align*}
as desired. This completes the proof. \qed

Setting $x,u,z\rightarrow 0$ and $y,v,w\rightarrow 1$, the above
theorem reduces to the following generating function
formula for the Rogers-Szeg\"{o} polynomials $g_n(x|q)$.

\begin{thm} We have
\begin{align}
&\sum_{n=0}^\infty\sum_{m=0}^\infty (-1)^{n+m}q^{{n\choose
2}+{m\choose 2}} g_{n+m}(a|q)g_n(b|q)g_m(c|q) \frac{t^n}{(q;q)_n}
\frac{s^m}{(q;q)_m}\nonumber \\[5pt]
&\quad=(q/a,at,abt,as,acs;q)_\infty\ {}_{4}\phi_1
\left(\begin{array}{c}
q/at,q/abt,q/as,q/acs\\
q/a
\end{array};q,\frac{a^3bct^2s^2}{q^3}\right). \label{gmn}
\end{align}
\end{thm}

It should be noted that Cao has considered
the same generating function and
obtained a double summation formula,
see \cite[Theorem\,4.4]{cao}. Using similar arguments, we can
derive several other generating function formulas for products
of $U_n(x,y, a;q)$. The detailed proofs are omitted.

\begin{thm} Assume $\max\{|xq/a|,|aut|\}<1$. We have
\begin{align*}
\sum_{n=0}^\infty & (-1)^{n+m}q^{-{n+m\choose 2}}U_{n+m}(x,y,a;q)
U_n(u,v,b;q)\frac{t^n}{(q;q)_n}\\
&=\frac{(yq/a,abt,avt;q)_\infty}{(xq/a,aut;q)_\infty}\,a^m\,
{}_{3}\phi_2 \left(\begin{array}{c}
y/x,q/abt,q/avt\\
yq/a,q/aut
\end{array};q,\frac{xbvt}{uq^{m}}\right).
\end{align*}
\end{thm}

Letting $x,u\rightarrow 0, y,v\rightarrow 1$ and  applying the
transformation formula for ${}_2\phi_1$ series
\cite[Appendix\,III.2]{gasrah}, we are led to the
following formula due to Cao \cite[Theorem\,4.1]{cao},
\begin{align*}
\sum_{n=0}^\infty & (-1)^{n}q^{n\choose 2}g_{n+m}(a|q)
g_n(b|q)\frac{t^n}{(q;q)_n}\\
&=\frac{(abt,at,bt,t;q)_\infty}{(abt^2/q;q)_\infty}
\frac{(q/t;q)_m}{(bt/q)^m(q^2/abt^2;q)_m} {}_{2}\phi_1
\left(\begin{array}{c}
q^{-m},q/abt\\
tq^{-m}
\end{array};q,bt\right).
\end{align*}

 \begin{thm} Assume $\max\{|xq/a|,|uq/b|\}<1$. We have
 \begin{align*}
\sum_{n=0}^\infty & \sum_{m=0}^\infty \sum_{k=0}^\infty (-1)^{k}q^{-{k\choose 2}-(n+m)k}
U_{n+k}(x,y,a;q)U_{m+k}(u,v,b;q)\frac{t^n}{(q;q)_n}\frac{s^m}{(q;q)_m}\frac{z^k}{(q;q)_k}\\
&=\frac{(yq/a,vq/b,bs,at,abz;q)_\infty}{(xq/a,uq/b;q)_\infty}\sum_{k=0}^\infty
\frac{(-1)^kq^{-{k\choose
2}}(y/x,q/at,q/abz;q)_k}{(q,yq/a;q)_k}\Big(\frac{xabtz}{q}\Big)^k\\
&\qquad\qquad\qquad\qquad\qquad\qquad\qquad\times {}_{3}\phi_1
\left(\begin{array}{c}
v/u,q/bs,q^{k+1}/abz\\
vq/b
\end{array};q,\frac{abusz}{q^{k+1}}\right).
\end{align*}
 \end{thm}

 Setting $x,u\rightarrow 0, y,v\rightarrow 1$ , we obtain
 \begin{align*}
\sum_{n=0}^\infty & \sum_{m=0}^\infty \sum_{k=0}^\infty
(-1)^{n+m+k}q^{{n\choose 2}+{m\choose 2}+{k\choose 2}}
g_{n+k}(a|q)g_{m+k}(b|q)\frac{t^n}{(q;q)_n}\frac{s^m}{(q;q)_m}\frac{z^k}{(q;q)_k}\\
&=(q/a,q/b,bs,at,abz;q)_\infty\sum_{k=0}^\infty
\frac{(q/at,q/abz;q)_k}{(q,q/a;q)_k}\Big(\frac{abtz}{q}\Big)^k
{}_{2}\phi_1 \left(\begin{array}{c}
q/bs,q^{k+1}/abz\\
q/b
\end{array};q,\frac{absz}{q^{k+1}}\right).
\end{align*}
Because of the convergence requirement, we should assume that
$q/at=q^{-r}$ and $|absz/q^{r+1}|<1$. Under this condition,  the
${}_2\phi_1$ series in the above expression can be summed by the
$q$-Gauss formula \cite[Appendix\,(II.8)]{gasrah}. It follows
that
\begin{align*}
\sum_{n=0}^\infty & \sum_{m=0}^\infty \sum_{k=0}^\infty
(-1)^{n+m+k}q^{{n\choose 2}+{m\choose 2}+{k\choose 2}}
g_{n+k}(a|q)g_{m+k}(b|q)\frac{t^n}{(q;q)_n}\frac{s^m}{(q;q)_m}\frac{z^k}{(q;q)_k}\\
&=
\frac{(q/a,s,at,az,bs,abz;q)_\infty}{(absz/q;q)_\infty}{}_{3}\phi_2
\left(\begin{array}{c}
q^{-n},q/az,q/abz\\
q/a,q^2/absz
\end{array};q,\frac{atz}{s}\right).
\end{align*}
It should mentioned that the terminating condition $t/v=q^{-r}$
is overlooked in the operator identity of Zhang and Wang \cite[Theorem\,2.5]{zhangwang05}, and the same condition is required
concerning the identity of Cao  \cite[Eq.\,(2.9)]{cao}.

 \begin{thm} Assume $\max\{|xq/a|,|vq/b|\}<1$. We have
 \begin{align*}
\sum_{k=0}^\infty & (-1)^{m+n+k}q^{-{n+k\choose 2}-{m\choose 2}-mk}
U_{n+k}(x,y,a;q)U_{m+k}(u,v,b;q)\frac{z^k}{(q;q)_k}\\
&=\frac{(yq/a,vq/b,abz;q)_\infty}{(xq/a,uq/b;q)_\infty} a^n b^m
\sum_{k=0}^\infty \frac{(y/x,q/abz;q)_k}{(q,yq/a;q)_k}
\Big(\frac{xbz}{q^n}\Big)^k {}_{2}\phi_1 \left(\begin{array}{c}
v/u,q^{k+1}/abz\\
vq/b
\end{array};q,\frac{auz}{q^{m+k}}\right).
\end{align*}
 \end{thm}

 Setting $x,u\rightarrow 0, y,v\rightarrow 1$, we deduce  that
  \begin{align*}
\sum_{k=0}^\infty & (-1)^{k}q^{{k\choose 2}}
g_{n+k}(a|q)g_{m+k}(b|q)\frac{z^k}{(q;q)_k}\\
&=(q/a,q/b,abz;q)_\infty a^n b^m \sum_{k=0}^\infty
\frac{(-1)^kq^{k\choose 2}(q/abz;q)_k}{(q,q/a;q)_k}
\Big(\frac{bz}{q^n}\Big)^k {}_{1}\phi_1 \left(\begin{array}{c}
q^{k+1}/abz\\
q/b
\end{array};q,\frac{az}{q^{m+k}}\right),
\end{align*}
which can be deduced from the formula of Cao
\cite[Theorem\,4.3]{cao} by three transformations, namely,  the the
limiting case of \cite[Appendix\,(III.2)]{gasrah} when $c\rightarrow
0$, the two transformations \cite[Appendix\,(III.2)]{gasrah}  and
 \cite[Appendix\,(III.7)]{gasrah}.

\vspace{.2cm} \noindent{\bf Acknowledgments.} This work was
supported by the 973 Project, the PCSIRT Project of the Ministry of
Education,  and the National
Science Foundation of China.

\end{document}